\documentclass[preprint,12pt]{elsarticle}

\usepackage[utf8]{inputenc}
\usepackage{amsmath,amssymb,amsthm,mathtools}
\usepackage{geometry}
\usepackage{xcolor}
\newtheorem{theorem}{Theorem}[section]
\newtheorem{lemma}[theorem]{Lemma}
\newtheorem{proposition}[theorem]{Proposition}
\newtheorem{definition}[theorem]{Definition}
\newtheorem{remark}[theorem]{Remark}

\newtheorem{corollary}[theorem]{Corollary}

\newcommand{\Pn}{\mathbb{P}_n}
\newcommand{\Hn}{\mathbb{H}_n}
\newcommand{\Tr}{\operatorname{Tr}}
\newcommand{\diag}{\operatorname{diag}}

\begin{document}
	
	\begin{frontmatter}
		
		
		
		\title{Weighted Two-Parameter Quantum Divergences Generated by Matrix Power-Mean Gaps}
		
		\author[label1,label2,label3]{Trung Hoa Dinh}
		\ead{hoa.dinhtrung@vlu.edu.vn}
		 \affiliation[label1]{organization={Laboratory for Applied and Industrial Mathematics, Institute for Computational Science and Artificial Intelligence, Van Lang University},
			             city={Ho Chi Minh City},
			             country={Vietnam}}
		\affiliation[label2]{organization={Faculty of Basic Sciences, Van Lang University},
			             city={Ho Chi Minh City},
			             country={Vietnam}}
		\affiliation[label3]{organization={Department of Mathematics and Statistics, Troy University},
						city={Troy},
			             postcode={36082},
			             state={AL},
			country={USA}}
		
		\author[d]{Minh Toan Ho} 
		\ead{hmtoan@math.ac.vn}
		\affiliation[d]{organization={Institute of Mathematics, VAST},
			addressline={18 Hoang Quoc Viet}, 
			city={Ha Noi},
			country={Vietnam}}

	\author[e]{Ha Trang Nguyen} 
	\ead{trangnh@utc.edu.vn}
\affiliation[e]{organization={Department of Mathematical Analysis, Faculty of Basic Sciences, University of Transport and Communications},
	addressline={3 Cau Giay Street, Lang Ward}, 
	city={Ha Noi},
	country={Vietnam}}
	
	\author[f,g]{Trung Dung Vuong} 
	\ead{vtdung@ptnk.edu.vn}
\affiliation[f]{organization={High School for the Gifted, VNUHCM},
	addressline={153 Nguyen Chi Thanh, An Dong Ward}, 
	city={Ho Chi Minh City},
	country={Vietnam}}
\affiliation[g]{organization={Vietnam National University Ho Chi Minh City},
	addressline={Linh Xuan Ward}, 
	city={Ho Chi Minh City},
	country={Vietnam}}
			
		\begin{abstract}
			Let $\Pn$ be the cone of $n\times n$ complex positive definite matrices.
			For $0<\theta<1$ and $r\ge0$, let $M_r^{(\theta)}$ denote the weighted
			matrix power mean, with the logarithmic endpoint at $r=0$, defined by
			\[
			M_r^{(\theta)}(A,B)
			=
			\begin{cases}
				\big((1-\theta)A^r+\theta B^r\big)^{1/r}, & r>0,\\[2mm]
				\exp\big((1-\theta)\log A+\theta\log B\big), & r=0.
			\end{cases}
			\]
			For $p>q\ge0$, we introduce
			\[
			\Phi_{p,q}^{(\theta)}(A,B)
			=
			\Tr\big[M_p^{(\theta)}(A,B)-M_q^{(\theta)}(A,B)\big].
			\]
			We prove that $\Phi_{p,q}^{(\theta)}$ is a quantum divergence in the
			sense of Bhatia--Gaubert--Jain. Moreover, for
			\[
			0\le q\le1\le p\le2,\qquad p>q,
			\]
			it is jointly convex and satisfies the data processing inequality for
			completely positive trace-preserving maps, with the output side interpreted
			by the regularized boundary convention whenever it is singular.
			
			The construction recovers Hellinger-type and Jensen--Shannon-type
			quantities as special cases or limits. We also compute the induced local
			quadratic form and provide scalar and matrix obstructions showing that the
			convexity and data-processing range cannot be extended to all parameters.
		\end{abstract}
		
		
	\end{frontmatter}
	
		
		\section{Introduction}\label{sec:introduction}
		Divergences between positive matrices and quantum states are fundamental
		objects in matrix analysis and quantum information theory.  They provide
		quantitative ways to compare noncommutative probability distributions, and
		their structural properties often reflect important information-theoretic
		principles such as positivity, distinguishability, convexity, and monotonicity
		under quantum channels.  Classical and quantum examples include relative
		entropy, Hellinger-type distances, Bures-type quantities, and
		Jensen--Shannon-type divergences.  Their analysis typically combines trace
		inequalities, matrix convexity, operator means, and the geometry of the cone
		of positive definite matrices; see, for instance,
		\cite{BhatiaMA,BhatiaPD,HiaiPetz,Watrous}.
		
		One systematic way to construct noncommutative divergences is to compare two
		matrix means by taking a trace gap.  This point of view was developed in a
		particularly transparent form by Bhatia, Gaubert, and Jain~\cite{BGJ}, who
		studied matrix versions of the Hellinger distance and introduced a natural
		local differential framework for quantum divergences.  Related constructions
		have since appeared in several directions: quantum Hellinger distances
		\cite{PitrikVirosztek}, divergences with \(p\)-power means \cite{LamLe},
		Bures--Wasserstein-type divergences involving matrix power means
		\cite{DinhLeVoVuong}, divergences associated with Kubo--Ando-type mean
		constructions \cite{DinhDuNguyenVuong}, divergences generated by
		monotonicity inequalities \cite{DinhLeOsakaPhan}, and trace-gap divergences
		associated with spectral geometric means and weighted spectral geometric
		means \cite{GanJeongKim,JeongKimTam}.  These results show that matrix means
		are not merely averaging procedures; they also form a flexible mechanism for
		producing noncommutative distances and information measures.
		
		The purpose of the present paper is to introduce and study a weighted
		two-parameter family of quantum divergences generated by trace gaps along the
		matrix power-mean scale.  Let \(\Pn\) denote the cone of \(n\times n\)
		positive definite complex matrices.  For \(0<\theta<1\) and \(r>0\), define
		the weighted matrix power mean
		\[
		M_r^{(\theta)}(A,B)
		=
		\bigl((1-\theta)A^r+\theta B^r\bigr)^{1/r},
		\qquad A,B\in\Pn,
		\]
		and define its logarithmic endpoint by
		\[
		M_0^{(\theta)}(A,B)
		=
		\exp\bigl((1-\theta)\log A+\theta\log B\bigr).
		\]
		For \(p>q\ge0\), we study the trace gap
		\[
		\Phi_{p,q}^{(\theta)}(A,B)
		=
		\Tr\bigl[M_p^{(\theta)}(A,B)-M_q^{(\theta)}(A,B)\bigr],
		\qquad A,B\in\Pn.
		\]
		This construction compares two different levels of the same weighted
		power-mean scale.  It is genuinely two-parameter, since both the upper and
		lower means are allowed to vary, and it includes the logarithmic endpoint
		\(q=0\).  The line \(p=1\) compares the weighted arithmetic mean with a
		lower-order power mean, while the symmetric case \(\theta=1/2\) contains the arithmetic-versus-power-mean family studied in \cite{LamLe}.  The present
		family is broader because it keeps both parameters \(p,q\) and the weight
		\(\theta\) free.
		
		Although the definition is expressed in terms of matrix means, the weighted
		power means used here are not, in general, Kubo--Ando means as functions of
		the pair \((A,B)\), except in special cases such as the arithmetic mean.
		Thus the positivity, convexity, and data-processing results obtained below
		do not follow directly from the standard Kubo--Ando theory of operator means
		\cite{KuboAndo}.  Instead, the paper relies on a different set of ideas:
		trace monotonicity along the matrix power-mean scale, the Carlen--Lieb
		convexity and concavity theory for Minkowski-type trace functionals
		\cite{CarlenLiebI,CarlenLiebII}, and an explicit second-order analysis on
		the diagonal.  This distinguishes the present construction from both
		Kubo--Ando trace gaps and the spectral-geometric-mean divergences studied in
		\cite{GanJeongKim,JeongKimTam}.
		
		The first main result, proved in Theorem~\ref{thm:main-divergence}, is a
		full-parameter divergence theorem.  We show that for every \(0<\theta<1\)
		and every \(p>q\ge0\), the functional
		\[
		\Phi_{p,q}^{(\theta)}(A,B)
		=
		\Tr\bigl[M_p^{(\theta)}(A,B)-M_q^{(\theta)}(A,B)\bigr]
		\]
		is a quantum divergence on \(\Pn\) in the sense of
		Bhatia--Gaubert--Jain~\cite{BGJ}.  In particular, it is nonnegative, it
		vanishes if and only if \(A=B\), its first derivative in the second variable
		vanishes on the diagonal, and its second derivative in the second variable is
		positive definite on the diagonal.  The proof of nonnegativity uses a trace
		comparison along the power-mean scale, while the strict equality condition
		comes from the strict equality case in Jensen's operator inequality for
		\(t^\alpha\), \(0<\alpha<1\).
		
		A key feature of the proof is that the local quadratic form induced by
		\(\Phi_{p,q}^{(\theta)}\) can be computed explicitly.  If
		\(A=\diag(a_1,\ldots,a_n)\) and \(H=(h_{ij})\in\Hn\), then
		\[
		D_2^2\Phi_{p,q}^{(\theta)}(A,A)[H,H]
		=
		4\theta(1-\theta)
		\left[
		\sum_i \frac{p-q}{4a_i}|h_{ii}|^2
		+
		\sum_{i<j}
		\bigl(\Gamma_p(a_i,a_j)-\Gamma_q(a_i,a_j)\bigr)|h_{ij}|^2
		\right],
		\]
		where, for \(r>0\) and \(a\ne b\),
		\[
		\Gamma_r(a,b)
		=
		\frac{a^{1-r}b^r+b^{1-r}a^r-a-b}{2r(a-b)^2},
		\]
		with the corresponding logarithmic endpoint at \(r=0\).  The positivity of
		this Hessian is reduced to the strict convexity, in the parameter \(r\), of
		the scalar function
		\[
		r\longmapsto x^r+x^{1-r}.
		\]
		Thus the construction gives not only a global trace-gap divergence, but also
		a completely explicit local Riemannian-type quadratic form on the positive
		cone.
		
		The second main result concerns joint convexity and data processing.  In the
		Carlen--Lieb parameter range
		\[
		0\le q\le1\le p\le2,\qquad p>q,
		\]
		we prove in Theorem~\ref{thm:joint-convexity} that
		\(\Phi_{p,q}^{(\theta)}\) is jointly convex on \(\Pn\times\Pn\).  Indeed,
		the upper trace functional
		\[
		(A,B)\longmapsto \Tr M_p^{(\theta)}(A,B)
		\]
		is jointly convex for \(1\le p\le2\), while the lower trace functional is
		jointly concave for \(0\le q\le1\).  Their difference is therefore jointly
		convex.  Using this convexity together with the standard
		dilation--partial-trace argument for quantum channels, we prove in
		Theorem~\ref{thm:dpi} that, after the regularized boundary convention
		described in Section~\ref{sec:convexity-dpi},
		\[
		\Phi_{p,q}^{(\theta)}(\Lambda(A),\Lambda(B))
		\le
		\Phi_{p,q}^{(\theta)}(A,B)
		\]
		for every completely positive trace-preserving map \(\Lambda\).  The
		boundary convention is needed because Stinespring dilation naturally
		produces singular operators on the larger system, especially when the
		logarithmic endpoint \(q=0\) is present.
		
		The family also contains several familiar quantities as special cases or
		limits.  The point \((p,q)=(1,1/2)\) gives
		\[
		\Phi_{1,1/2}^{(\theta)}(A,B)
		=
		\theta(1-\theta)\Tr(A^{1/2}-B^{1/2})^2,
		\]
		which is a weighted square-root Hellinger trace gap; compare
		\cite{BGJ,PitrikVirosztek}.  The point \((p,q)=(1,0)\) gives a
		log-Euclidean Hellinger-type trace gap.  Moreover, the diagonal limit
		\[
		\lim_{\varepsilon\downarrow0}
		\frac{\Phi_{1+\varepsilon,1}^{(\theta)}(A,B)}{\varepsilon}
		\]
		recovers the weighted quantum Jensen--Shannon divergence, a quantity studied
		in \cite{BrietHarremoes2009,LambertiQJSD}.  At the logarithmic endpoint we
		also compute
		\[
		\lim_{\varepsilon\downarrow0}
		\frac{\Phi_{\varepsilon,0}^{(\theta)}(A,B)}{\varepsilon}
		=
		\frac{\theta(1-\theta)}{2}
		\Tr\left[
		M_0^{(\theta)}(A,B)(\log A-\log B)^2
		\right].
		\]
		These formulas place Hellinger-type, log-Euclidean, and
		Jensen--Shannon-type quantities inside one weighted power-mean framework.
		
		Finally, we examine how far the convexity and data-processing range can be
		extended.  In Section~\ref{sec:sharpness}, scalar restrictions show that
		joint convexity fails for every \(1<q<p\).  For \(p>2\), failure already
		occurs at \(q=1\), as a consequence of the nonconvexity of the
		Carlen--Lieb trace functional outside its convexity range
		\cite{CarlenLiebI,CarlenLiebII}.  In the remaining scalar region
		\(0<q<p<1\), we derive an exact one-dimensional criterion for convexity of
		the scalar restriction.  At the logarithmic endpoint \(q=0\), this criterion
		becomes explicit: the scalar restriction is convex if and only if
		\[
		p\ge \max\{\theta,1-\theta\}.
		\]
		Equivalently, scalar convexity fails precisely when
		\[
		0<p<\max\{\theta,1-\theta\}.
		\]
		These obstruction results show that the range \(0\le q\le1\le p\le2\) is
		natural for joint convexity and data processing, while the full matrix-level
		classification in the sublinear region \(0<q<p<1\) remains open.
		
		The paper is organized as follows.  Section~\ref{sec:preliminaries} fixes
		notation and definitions.  Section~\ref{sec:divergence-proof} proves the
		full-parameter divergence theorem for \(\Phi_{p,q}^{(\theta)}\), including
		the explicit diagonal Hessian formula.  Section~\ref{sec:limits-special-cases} discusses special cases, limiting procedures, parameter monotonicity, and the
		induced local quadratic form.  Section~\ref{sec:convexity-dpi} proves joint
		convexity and the data processing inequality in the range
		\(0\le q\le1\le p\le2\).  Section~\ref{sec:sharpness} gives scalar and
		matrix obstructions outside this range.

		\section{Preliminaries}\label{sec:preliminaries}
		
		Let $\Pn$ denote the cone of $n\times n$ positive definite complex
		matrices, and let $\Hn$ denote the real vector space of Hermitian
		matrices. For $A\in\Pn$ and $r\in\mathbb R$, the power $A^r$ is defined
		by the usual functional calculus.
		
		\begin{definition}
			Let \(0<\theta<1\). For \(r>0\), define the weighted matrix power mean
			\[
			M_r^{(\theta)}(A,B)
			=
			\big((1-\theta)A^r+\theta B^r\big)^{1/r},
			\qquad A,B\in\Pn,
			\]
			and define its logarithmic endpoint by
			\[
			M_0^{(\theta)}(A,B)
			=
			\exp\big((1-\theta)\log A+\theta\log B\big).
			\]
			We write
			\[
			T_r^{(\theta)}(A,B)=\Tr M_r^{(\theta)}(A,B).
			\]
			For \(p>q\ge0\), set
			\[
			\Phi_{p,q}^{(\theta)}(A,B)
			=
			T_p^{(\theta)}(A,B)-T_q^{(\theta)}(A,B)
			=
			\Tr\big[M_p^{(\theta)}(A,B)-M_q^{(\theta)}(A,B)\big].
			\]
			When \(\theta=1/2\), we abbreviate
			\[
			M_r=M_r^{(1/2)},\qquad
			\Phi_{p,q}=\Phi_{p,q}^{(1/2)}.
			\]
		\end{definition}
		
		\begin{definition}
			A smooth function $\Phi:\Pn\times\Pn\to\mathbb R$ is called a quantum
			divergence in the sense of Bhatia--Gaubert--Jain if the following
			conditions hold:
			\begin{enumerate}
				\item $\Phi(A,B)\ge0$ for all $A,B\in\Pn$, and
				\[
				\Phi(A,B)=0\quad\Longleftrightarrow\quad A=B.
				\]
				\item For every $A\in\Pn$ and every $H\in\Hn$,
				\[
				D_2\Phi(A,A)[H]=0.
				\]
				\item For every $A\in\Pn$ and every nonzero $H\in\Hn$,
				\[
				D_2^2\Phi(A,A)[H,H]>0.
				\]
			\end{enumerate}
			Here $D_2$ and $D_2^2$ denote the first and second Fréchet derivatives
			with respect to the second variable. 
		\end{definition}
		
		\section{Proof of the divergence theorem}\label{sec:divergence-proof}
		
		This section proves that \(\Phi_{p,q}^{(\theta)}\) is a quantum divergence
		in the sense of Bhatia--Gaubert--Jain. The proof is divided into three
		parts: global nonnegativity with a strict equality condition, vanishing of
		the first derivative on the diagonal, and positivity of the diagonal
		Hessian.
		
		\subsection{Nonnegativity and equality condition}
		
		We first prove the global positivity property. The key point is that no
		operator monotonicity of $x^{1/q}$ is required; the trace comparison is
		obtained from eigenvalue monotonicity.
		
		\begin{proposition}\label{prop:nonneg}
			Let $0<\theta<1$ and let $p>q\ge0$. Then
			\[
			\Phi_{p,q}^{(\theta)}(A,B)\ge0
			\]
			for all $A,B\in\Pn$. Moreover,
			\[
			\Phi_{p,q}^{(\theta)}(A,B)=0
			\quad\Longleftrightarrow\quad
			A=B.
			\]
		\end{proposition}
		
		\begin{proof}
			We first treat the case \(p>q>0\). Put
			\[
			\alpha=\frac{q}{p}\in(0,1),\qquad
			X=A^p,\qquad Y=B^p,
			\]
			and
			\[
			C=(1-\theta)X+\theta Y,\qquad
			D=(1-\theta)X^\alpha+\theta Y^\alpha
			=(1-\theta)A^q+\theta B^q.
			\]
			By the operator concavity of \(t^\alpha\), \(0<\alpha<1\), and the
			strict equality case in the two-point Jensen inequality
			\cite{BhatiaMA,BhatiaPD,HansenPedersen},
			\[
			D\le C^\alpha,
			\]
			with equality if and only if \(X=Y\).
			
			Since \(D\le C^\alpha\), Weyl monotonicity of eigenvalues
			\cite[p.~63]{BhatiaMA} gives
			\[
			\lambda_j(D)\le \lambda_j(C^\alpha),\qquad j=1,\ldots,n.
			\]
			Because the scalar function \(x\mapsto x^{1/q}\) is strictly increasing,
			\[
			\Tr D^{1/q}
			=
			\sum_{j=1}^n \lambda_j(D)^{1/q}
			\le
			\sum_{j=1}^n \lambda_j(C^\alpha)^{1/q}
			=
			\Tr (C^\alpha)^{1/q}
			=
			\Tr C^{1/p}.
			\]
			Thus
			\[
			T_q^{(\theta)}(A,B)\le T_p^{(\theta)}(A,B),
			\]
			and hence \(\Phi_{p,q}^{(\theta)}(A,B)\ge0\).
			
			If \(A=B\), then \(M_p^{(\theta)}(A,A)=M_q^{(\theta)}(A,A)=A\), so
			\(\Phi_{p,q}^{(\theta)}(A,A)=0\).
			
			Conversely, assume that \(\Phi_{p,q}^{(\theta)}(A,B)=0\). Then equality
			holds in the above trace comparison. Hence
			\[
			\lambda_j(D)=\lambda_j(C^\alpha),\qquad j=1,\ldots,n.
			\]
			This implies
			\[
			\Tr(C^\alpha-D)=0,
			\]
			and $C^\alpha-D\ge0$, therefore \(D=C^\alpha\). Thus equality holds in the two-point
			Jensen inequality for \(t^\alpha\), and the strict equality case gives
			\(X=Y\). Hence \(A^p=B^p\), and by injectivity of the functional
			calculus on \(\Pn\), \(A=B\).
			
			It remains to treat the endpoint \(q=0\). For \(r>0\), write
			\[
			S_r(A,B)=T_r^{(\theta)}(A,B).
			\]
			Indeed,
			\[
			S_s(A,B)-S_r(A,B)=\Phi_{s,r}^{(\theta)}(A,B),
			\]
			and the case \(s>r>0\) already proved gives
			\(\Phi_{s,r}^{(\theta)}(A,B)\ge0\), with equality if and only if \(A=B\). Moreover,
			\[
			\lim_{r\downarrow0}S_r(A,B)=S_0(A,B),
			\]
			because
			\[
			\lim_{r\downarrow0}
			\big((1-\theta)A^r+\theta B^r\big)^{1/r}
			=
			\exp\big((1-\theta)\log A+\theta\log B\big).
			\]
			Letting \(r\downarrow0\) in \(S_r(A,B)\le S_p(A,B)\) gives
			\[
			S_0(A,B)\le S_p(A,B),
			\]
			that is,
			\[
			\Phi_{p,0}^{(\theta)}(A,B)\ge0.
			\]
			If equality holds, then for every \(r\in(0,p)\),
			\[
			S_0(A,B)\le S_r(A,B)\le S_p(A,B)=S_0(A,B).
			\]
			Thus \(S_r(A,B)=S_p(A,B)\) for some \(0<r<p\). By the strict equality
			case already proved for positive parameters, \(A=B\).
		\end{proof}
		
		\subsection{The first derivative on the diagonal}
		
		We now verify the vanishing of the first derivative on the diagonal.
		
		\begin{proposition}\label{prop:first-derivative}
			Let $0<\theta<1$ and let $r\ge0$. Then, for every $A\in\Pn$ and every
			$H\in\Hn$,
			\[
			D_2T_r^{(\theta)}(A,A)[H]
			=
			\theta\,\Tr H.
			\]
			Consequently, for $p>q\ge0$,
			\[
			D_2\Phi_{p,q}^{(\theta)}(A,A)[H]=0.
			\]
		\end{proposition}
		
		\begin{proof}
			First assume $r>0$. Fix $A\in\Pn$ and $H\in\Hn$, and put
			\[
			B(s)=A+sH.
			\]
			For $|s|$ sufficiently small, $B(s)\in\Pn$. Define
			\[
			\psi_r(s)
			=
			T_r^{(\theta)}(A,A+sH)
			=
			\Tr\left[
			\big((1-\theta)A^r+\theta(A+sH)^r\big)^{1/r}
			\right].
			\]
			Let
			\[
			X_r(s)=(1-\theta)A^r+\theta(A+sH)^r.
			\]
			Then
			\[
			X_r(0)=A^r,
			\qquad
			X_r'(0)=\theta\,D(A^r)[H].
			\]
			By the chain rule for trace functions,
			\[
			\psi_r'(0)
			=
			\Tr\left[
			\frac1r (A^r)^{1/r-1}\,\theta D(A^r)[H]
			\right]
			=
			\frac{\theta}{r}\Tr\left[A^{1-r}D(A^r)[H]\right].
			\]
			Differentiating the identity
			\[
			\Tr(A+sH)
			=
			\Tr\big(((A+sH)^r)^{1/r}\big)
			\]
			at $s=0$ gives
			\[
			\Tr H
			=
			\frac1r\Tr\left[A^{1-r}D(A^r)[H]\right].
			\]
			Therefore
			\[
			\psi_r'(0)=\theta\,\Tr H.
			\]
			
			For $r=0$, we have
			\[
			\psi_0(s)
			=
			\Tr\exp\big((1-\theta)\log A+\theta\log(A+sH)\big).
			\]
			Using the trace derivative of the exponential,
			\[
			\frac{d}{ds}\Tr e^{L(s)}\bigg|_{s=0}
			=
			\Tr\big(e^{L(0)}L'(0)\big),
			\]
			with
			\[
			L(s)=(1-\theta)\log A+\theta\log(A+sH),
			\]
			we get
			\[
			\psi_0'(0)
			=
			\theta\,\Tr\big(A\,D(\log A)[H]\big).
			\]
			But differentiating
			\[
			\Tr(A+sH)=\Tr\exp(\log(A+sH))
			\]
			at $s=0$ gives
			\[
			\Tr H=\Tr\big(A\,D(\log A)[H]\big).
			\]
			Thus
			\[
			\psi_0'(0)=\theta\,\Tr H.
			\]
			
			Applying the formula to $r=p$ and $r=q$ gives
			\[
			D_2\Phi_{p,q}^{(\theta)}(A,A)[H]
			=
			\theta\Tr H-\theta\Tr H=0.
			\]
		\end{proof}
		
		\subsection{The second derivative on the diagonal}
		
		We now compute the Hessian explicitly and prove that it is positive
		definite.
		
		Before the computation, we recall the divided-difference notation used below.
		Let $f$ be a smooth scalar function on $(0,\infty)$. Its first divided
		difference is
		\[
		f^{[1]}(a,b)
		=
		\begin{cases}
			\dfrac{f(a)-f(b)}{a-b}, & a\ne b,\\[1.2ex]
			f'(a), & a=b.
		\end{cases}
		\]
		The second divided difference is defined recursively by
		\[
		f^{[2]}(a,b,c)
		=
		\frac{f^{[1]}(a,b)-f^{[1]}(b,c)}{a-c},
		\]
		when $a\ne c$, with the value at repeated points defined by continuity.
		
		We shall use the standard Daleckii--Krein divided-difference expansion
		for matrix functions; see, for example, \cite{BhatiaMA,HiaiPetz}. Namely,
		\[
		f(A+sH)
		=
		f(A)+sDf(A)[H]+s^2Q_f(A;H)+O(s^3),
		\]
		where, if $A=\diag(a_1,\ldots,a_n)$, then
		\[
		(Df(A)[H])_{ij}=f^{[1]}(a_i,a_j)h_{ij},
		\]
		and
		\[
		(Q_f(A;H))_{ii}
		=
		\sum_{j=1}^n f^{[2]}(a_i,a_j,a_i)|h_{ij}|^2.
		\]
		
		We next introduce the scalar coefficient that appears in the off-diagonal
		part of the Hessian. For $r>0$ and $a,b>0$, define
		\[
		\Gamma_r(a,b)
		=
		\frac{a^{1-r}b^r+b^{1-r}a^r-a-b}{2r(a-b)^2},
		\qquad a\ne b,
		\]
		and
		\[
		\Gamma_r(a,a)=\frac{r-1}{2a}.
		\]
		For $r=0$, define
		\[
		\Gamma_0(a,b)
		=
		\frac{\log(b/a)}{2(a-b)},
		\qquad a\ne b,
		\]
		and
		\[
		\Gamma_0(a,a)=-\frac1{2a}.
		\]
		These definitions are precisely the limits of \(\Gamma_r(a,b)\) as
		\(r\downarrow0\), both for \(a\ne b\) and for \(a=b\).
		
		\begin{lemma}\label{lem:C2-limit-r0}
			Let $A\in\Pn$, $H\in\Hn$, and $0<\theta<1$. For $r>0$, set
			\[
			\psi_r(s)
			=
			\Tr\left[
			\big((1-\theta)A^r+\theta(A+sH)^r\big)^{1/r}
			\right],
			\]
			and set
			\[
			\psi_0(s)
			=
			\Tr\exp\big((1-\theta)\log A+\theta\log(A+sH)\big).
			\]
			Then
			\[
			\psi_r''(0)\longrightarrow \psi_0''(0)
			\qquad (r\downarrow0).
			\]
			Consequently,
			\[
			D_2^2T_0^{(\theta)}(A,A)[H,H]
			=
			\lim_{r\downarrow0}D_2^2T_r^{(\theta)}(A,A)[H,H].
			\]
		\end{lemma}
		
		\begin{proof}
			For $s$ sufficiently close to $0$, the matrices $A+sH$ remain in a fixed
			compact subset of $\Pn$. Hence the map
			\[
			s\longmapsto \log(A+sH)
			\]
			is $C^2$ in operator norm, and the maps
			\[
			(r,s)\longmapsto A^r=\exp(r\log A),
			\qquad
			(r,s)\longmapsto (A+sH)^r=\exp(r\log(A+sH))
			\]
			are $C^2$ in $s$ and analytic in $r$ near $(0,0)$. Put
			\[
			L(s)=(1-\theta)\log A+\theta\log(A+sH).
			\]
			Taylor expansion of the exponential, uniformly in the $C^2$-topology with
			respect to $s$, gives
			\[
			(1-\theta)A^r+\theta(A+sH)^r
			=
			I+rL(s)+r^2R(r,s),
			\]
			where $R(r,s)$ is locally bounded in $C^2$ as $r\downarrow0$.
			For small $r$, the spectrum of $I+rL(s)+r^2R(r,s)$ stays in a fixed
			neighbourhood of $1$. The analytic expansion of the logarithm at the
			identity therefore gives, again in the $C^2$-topology in $s$,
			\[
			\frac1r\log\big(I+rL(s)+r^2R(r,s)\big)
			=
			L(s)+r\widetilde R(r,s),
			\]
			where $\widetilde R(r,s)$ is locally bounded in $C^2$. Consequently
			\[
			\left(I+rL(s)+r^2R(r,s)\right)^{1/r}
			=
			\exp\big(L(s)+r\widetilde R(r,s)\big)
			\]
			converges to $\exp(L(s))$ in $C^2$ with respect to $s$. Taking traces
			gives $\psi_r''(0)\to\psi_0''(0)$.
		\end{proof}
		
		\begin{proposition}\label{prop:hessian-Tr}
			Let $0<\theta<1$, $r\ge0$, and $A\in\Pn$. Choose an orthonormal basis in
			which
			\[
			A=\diag(a_1,\ldots,a_n),\qquad a_i>0.
			\]
			Let $H=(h_{ij})\in\Hn$. Then
			\[
			D_2^2T_r^{(\theta)}(A,A)[H,H]
			=
			4\theta(1-\theta)
			\left[
			\sum_i\frac{r-1}{4a_i}|h_{ii}|^2
			+
			\sum_{i<j}\Gamma_r(a_i,a_j)|h_{ij}|^2
			\right],
			\]
			where for $r=0$ the diagonal coefficient is interpreted as
			\[
			\frac{r-1}{4a_i}\bigg|_{r=0}=-\frac1{4a_i}.
			\]
		\end{proposition}
		
		\begin{proof}
			We first assume $r>0$. Let
			\[
			f_r(x)=x^r,\qquad g_r(x)=x^{1/r}.
			\]
			Put $B(s)=A+sH$ and
			\[
			\psi_r(s)=T_r^{(\theta)}(A,B(s)).
			\]
			Using the second-order expansion of $f_r(A+sH)$, write
			\[
			(A+sH)^r
			=
			A^r+sL_r+s^2Q_r+O(s^3),
			\]
			where
			\[
			(L_r)_{ij}=f_r^{[1]}(a_i,a_j)h_{ij},
			\]
			and
			\[
			(Q_r)_{ii}
			=
			\sum_j f_r^{[2]}(a_i,a_j,a_i)|h_{ij}|^2.
			\]
			Therefore
			\[
			(1-\theta)A^r+\theta(A+sH)^r
			=
			A^r+s\theta L_r+s^2\theta Q_r+O(s^3).
			\]
			Set
			\[
			C=A^r,\qquad Y=\theta L_r,\qquad Z=\theta Q_r.
			\]
			Then
			\[
			\psi_r(s)
			=
			\Tr g_r(C+sY+s^2Z+O(s^3)).
			\]
			The second-order trace expansion gives
			\[
			\psi_r(s)
			=
			\Tr g_r(C)
			+
			s\Tr Dg_r(C)[Y]
			+
			s^2
			\left\{
			\Tr Dg_r(C)[Z]
			+
			\sum_{i,j}g_r^{[2]}(c_i,c_j,c_i)|Y_{ij}|^2
			\right\}
			+
			O(s^3),
			\]
			where $c_i=a_i^r$. Hence
			\[
			\frac12\psi_r''(0)
			=
			\Tr Dg_r(C)[Z]
			+
			\sum_{i,j}g_r^{[2]}(c_i,c_j,c_i)|Y_{ij}|^2.
			\]
			
			For the diagonal entries, a direct substitution gives the coefficient of
			$|h_{ii}|^2$ in $\psi_r''(0)$:
			\[
			\theta(1-\theta)\frac{r-1}{a_i}
			=
			4\theta(1-\theta)\frac{r-1}{4a_i}.
			\]
			
			Now consider an off-diagonal pair $i<j$ and write $a=a_i$, $b=a_j$. We first treat the case $a\ne b$. The case $a=b$ follows by continuity of divided differences and of the resulting Hessian coefficient; indeed, the limit $b\to a$ gives
			\[
			\Gamma_r(a,a)=\frac{r-1}{2a}.
			\]
			Let
			\[
			\mathcal A_r(a,b)
			=
			g_r'(a^r)f_r^{[2]}(a,b,a)
			+
			g_r'(b^r)f_r^{[2]}(b,a,b),
			\]
			and
			\[
			\mathcal B_r(a,b)
			=
			\big(f_r^{[1]}(a,b)\big)^2
			\left[
			g_r^{[2]}(a^r,b^r,a^r)
			+
			g_r^{[2]}(b^r,a^r,b^r)
			\right].
			\]
			The coefficient of $|h_{ij}|^2$, grouped over the pair $i<j$, in
			$\psi_r''(0)$ is
			\[
			2\theta\,\mathcal A_r(a,b)+2\theta^2\,\mathcal B_r(a,b).
			\]
			For completeness, we record the algebra. One has
			\[
			f_r^{[1]}(a,b)=\frac{a^r-b^r}{a-b},
			\]
			\[
			f_r^{[2]}(a,b,a)
			=
			\frac{r a^{r-1}(a-b)-(a^r-b^r)}{(a-b)^2},
			\]
			and
			\[
			g_r^{[2]}(a^r,b^r,a^r)
			=
			\frac{\frac1r a^{1-r}(a^r-b^r)-(a-b)}
			{(a^r-b^r)^2}.
			\]
			The analogous formulas with $a$ and $b$ interchanged give, after
			substitution,
			\[
			\mathcal A_r(a,b)
			=
			\frac{a^{1-r}b^r+b^{1-r}a^r-a-b}{r(a-b)^2},
			\]
			and
			\[
			\mathcal B_r(a,b)
			=
			-\frac{a^{1-r}b^r+b^{1-r}a^r-a-b}{r(a-b)^2}.
			\]
			Therefore
			\[
			\mathcal A_r(a,b)=2\Gamma_r(a,b),
			\qquad
			\mathcal B_r(a,b)=-2\Gamma_r(a,b).
			\]
			Therefore the off-diagonal coefficient is
			\[
			2\theta(2\Gamma_r(a,b))
			+
			2\theta^2(-2\Gamma_r(a,b))
			=
			4\theta(1-\theta)\Gamma_r(a,b).
			\]
			This proves the formula for $r>0$.
			
			The case $r=0$ follows from Lemma~\ref{lem:C2-limit-r0} by taking the
			limit $r\downarrow0$ in the already established formula. The coefficients
			converge to
			\[
			-\frac1{4a_i}
			\quad\text{and}\quad
			\Gamma_0(a,b)
			=
			\frac{\log(b/a)}{2(a-b)}.
			\]
		\end{proof}
		
		The preceding formula reduces positivity of the Hessian of
		$\Phi_{p,q}^{(\theta)}$ to a scalar monotonicity property of
		$\Gamma_r$. This gives the following positivity result.
		
		\begin{theorem}\label{thm:hessian-positive}
			Let \(0<\theta<1\) and \(p>q\ge0\). Then, for every \(A\in\Pn\),
			the diagonal Hessian
			\[
			H\longmapsto D_2^2\Phi_{p,q}^{(\theta)}(A,A)[H,H]
			\]
			is positive definite on \(\Hn\). Equivalently, for every
			\(H\in\Hn\),
			\[
			D_2^2\Phi_{p,q}^{(\theta)}(A,A)[H,H]\ge0,
			\]
			with equality if and only if \(H=0\).
		\end{theorem}
		
		\begin{proof}
			Choose an orthonormal basis in which
			\[
			A=\diag(a_1,\ldots,a_n),\qquad a_i>0.
			\]
			By Proposition~\ref{prop:hessian-Tr},
			\[
			D_2^2\Phi_{p,q}^{(\theta)}(A,A)[H,H]
			=
			4\theta(1-\theta)
			\left[
			\sum_i\frac{p-q}{4a_i}|h_{ii}|^2
			+
			\sum_{i<j}
			\big(\Gamma_p(a_i,a_j)-\Gamma_q(a_i,a_j)\big)|h_{ij}|^2
			\right],
			\]
			with the convention for $\Gamma_0$ stated above.
			
			The diagonal coefficients are strictly positive because $p>q$ and
			$a_i>0$. It remains to prove
			\[
			\Gamma_p(a,b)>\Gamma_q(a,b)
			\]
			for all $a,b>0$ and $p>q\ge0$.
			
			If $a=b$, then
			\[
			\Gamma_r(a,a)=\frac{r-1}{2a}
			\]
			for $r>0$, and the same formula extends to $r=0$. Hence
			\[
			\Gamma_p(a,a)-\Gamma_q(a,a)
			=
			\frac{p-q}{2a}>0.
			\]
			
			Now suppose $a\ne b$ and put
			\[
			x=\frac ba\ne1.
			\]
			For $r>0$,
			\[
			\Gamma_r(a,b)
			=
			\frac1{2a(1-x)^2}
			\cdot
			\frac{x^r+x^{1-r}-(1+x)}{r}.
			\]
			Define
			\[
			\eta_x(r)=x^r+x^{1-r}.
			\]
			Then
			\[
			\eta_x(0)=1+x,
			\]
			and
			\[
			\Gamma_r(a,b)
			=
			\frac1{2a(1-x)^2}
			\cdot
			\frac{\eta_x(r)-\eta_x(0)}{r}.
			\]
			Moreover,
			\[
			\eta_x''(r)
			=
			(\log x)^2(x^r+x^{1-r})>0,
			\]
			because $x\ne1$. Thus $\eta_x$ is strictly convex.
			
			For $0<q<p$, write
			\[
			q=\frac qp\,p+\left(1-\frac qp\right)0.
			\]
			Here the correct weight is
			\[
			\lambda=\frac qp\in(0,1).
			\]
			Strict convexity gives
			\[
			\eta_x(q)
			<
			\frac qp\,\eta_x(p)
			+
			\left(1-\frac qp\right)\eta_x(0).
			\]
			Subtracting $\eta_x(0)$ and dividing by $q>0$, we obtain
			\[
			\frac{\eta_x(q)-\eta_x(0)}{q}
			<
			\frac{\eta_x(p)-\eta_x(0)}{p}.
			\]
			Therefore
			\[
			\Gamma_q(a,b)<\Gamma_p(a,b).
			\]
			
			For $q=0$, the same conclusion follows by taking the limit
			\[
			\Gamma_0(a,b)
			=
			\frac1{2a(1-x)^2}\eta_x'(0),
			\]
			and using the fact that for a strictly convex function the secant slope
			\[
			\frac{\eta_x(p)-\eta_x(0)}{p}
			\]
			is strictly larger than the tangent slope $\eta_x'(0)$.
			
			Hence every coefficient in the Hessian expansion is strictly positive.
			If $H\ne0$, at least one entry $h_{ij}$ is nonzero, and the corresponding
			positive coefficient makes the whole sum strictly positive.
		\end{proof}

		Combining the global nonnegativity and equality condition in
		Proposition~\ref{prop:nonneg}, the first-variation identity in
		Proposition~\ref{prop:first-derivative}, and the positive definiteness of
		the diagonal Hessian in Theorem~\ref{thm:hessian-positive}, we obtain the
		announced divergence theorem.
		
		\begin{theorem}\label{thm:main-divergence}
			Let \(0<\theta<1\) and let \(p>q\ge0\). Define
			\[
			\Phi_{p,q}^{(\theta)}(A,B)
			=
			\Tr\big[M_p^{(\theta)}(A,B)-M_q^{(\theta)}(A,B)\big],
			\qquad A,B\in\Pn.
			\]
			Then \(\Phi_{p,q}^{(\theta)}\) is a quantum divergence in the sense of
			Bhatia--Gaubert--Jain.
		\end{theorem}
		\begin{corollary}\label{cor:symmetric-main}
			Let $p>q\ge0$. Then the symmetric trace gap
			\[
			\Phi_{p,q}(A,B)
			=
			\Tr\left[
			\left(\frac{A^p+B^p}{2}\right)^{1/p}
			-
			M_q(A,B)
			\right],
			\]
			where
			\[
			M_q(A,B)=
			\begin{cases}
				\left(\dfrac{A^q+B^q}{2}\right)^{1/q}, & q>0,\\[2ex]
				\exp\left(\dfrac{\log A+\log B}{2}\right), & q=0,
			\end{cases}
			\]
			is a quantum divergence on $\Pn$.
		\end{corollary}

		\section{Special cases, limits, and local form}\label{sec:limits-special-cases}
		
		We now record several special cases and limit processes contained in the
		two-parameter family
		\[
		\Phi_{p,q}^{(\theta)}(A,B)
		=
		T_p^{(\theta)}(A,B)-T_q^{(\theta)}(A,B),
		\qquad
		p>q\ge0.
		\]
		This section is useful for comparing the present family with earlier
		power-mean divergences and with standard quantum information quantities.
		In particular, the cases below recover the Lam--Le power-mean line, a
		constant multiple of the square-root Hellinger trace gap, the
		log-Euclidean Hellinger-type divergence, and the quantum Jensen--Shannon
		divergence through a diagonal parameter limit.
		
		\subsection{The arithmetic-versus-power-mean subfamily \(p=1\)}
		
		When \(p=1\), the upper mean becomes the weighted arithmetic mean:
		\[
		M_1^{(\theta)}(A,B)=(1-\theta)A+\theta B.
		\]
		Thus, for \(0\le q<1\),
		\[
		\Phi_{1,q}^{(\theta)}(A,B)
		=
		\Tr\left[
		(1-\theta)A+\theta B
		-
		M_q^{(\theta)}(A,B)
		\right].
		\]
		For \(0<q<1\), this gives
		\[
		\Phi_{1,q}^{(\theta)}(A,B)
		=
		\Tr\left[
		(1-\theta)A+\theta B
		-
		\big((1-\theta)A^q+\theta B^q\big)^{1/q}
		\right].
		\]
		In the symmetric case \(\theta=1/2\), this is the usual
		arithmetic-versus-\(q\)-power-mean trace gap studied in the power-mean
		divergence framework of Lam--Le~\cite{LamLe}.
		
		The endpoint \(q=0\) gives the log-Euclidean version:
		\[
		\Phi_{1,0}^{(\theta)}(A,B)
		=
		\Tr\left[
		(1-\theta)A+\theta B
		-
		\exp\big((1-\theta)\log A+\theta\log B\big)
		\right].
		\]
		In particular, in the symmetric case,
		\[
		\Phi_{1,0}^{(1/2)}(A,B)
		=
		\frac12
		\Tr\left[
		A+B
		-
		2\exp\left(\frac{\log A+\log B}{2}\right)
		\right].
		\]
		Thus \(\Phi_{1,0}^{(1/2)}\) is one half of the standard squared
		log-Euclidean Hellinger-type trace gap.
		\subsection{The square-root Hellinger special case}
		
		A particularly simple special case occurs when \(p=1\) and \(q=1/2\).
		The normalization below follows the trace-gap convention of this paper.
		In the symmetric case it is a constant multiple of the commonly used
		squared Hellinger-type trace expression; compare \cite{BGJ,PitrikVirosztek}.
		
		\begin{proposition}\label{prop:sqrt-hellinger}
			Let \(0<\theta<1\). Then
			\[
			\Phi_{1,1/2}^{(\theta)}(A,B)
			=
			\theta(1-\theta)\Tr\left(A^{1/2}-B^{1/2}\right)^2.
			\]
			In particular, in the symmetric case \(\theta=1/2\),
			\[
			\Phi_{1,1/2}^{(1/2)}(A,B)
			=
			\frac14\Tr\left(A^{1/2}-B^{1/2}\right)^2.
			\]
		\end{proposition}
		
		\begin{proof}
			We have
			\[
			M_{1/2}^{(\theta)}(A,B)
			=
			\left((1-\theta)A^{1/2}+\theta B^{1/2}\right)^2.
			\]
			Therefore
			\[
			\begin{aligned}
				\Phi_{1,1/2}^{(\theta)}(A,B)
				&=
				\Tr\left[
				(1-\theta)A+\theta B
				-
				\left((1-\theta)A^{1/2}+\theta B^{1/2}\right)^2
				\right] \\
				&=
				\theta(1-\theta)
				\Tr\left[
				A+B-A^{1/2}B^{1/2}-B^{1/2}A^{1/2}
				\right] \\
				&=
				\theta(1-\theta)
				\Tr\left(A^{1/2}-B^{1/2}\right)^2.
			\end{aligned}
			\]
		\end{proof}

		\subsection{Parameter limits}
		
		The preceding special cases are obtained by fixing the parameters. We now
		record two limiting procedures. The first differentiates the trace
		functional along the positive matrix power-mean scale and, at \(r=1\),
		recovers the weighted quantum Jensen--Shannon divergence \cite{BrietHarremoes2009}. The second
		describes the logarithmic endpoint obtained by letting \(p\downarrow0\)
		with \(q=0\).
		
		For \(r>0\), set
		\[
		C_r=(1-\theta)A^r+\theta B^r,
		\qquad
		M_r=C_r^{1/r}.
		\]
		Then
		\[
		T_r^{(\theta)}(A,B)=\Tr M_r.
		\]
		
		\begin{proposition}\label{prop:diagonal-parameter-derivative}
			Let \(r>0\). Then
			\[
			\lim_{\varepsilon\downarrow0}
			\frac{\Phi_{r+\varepsilon,r}^{(\theta)}(A,B)}{\varepsilon}
			=
			\frac{\partial}{\partial r}T_r^{(\theta)}(A,B),
			\]
			where
			\[
			\begin{aligned}
				\frac{\partial}{\partial r}T_r^{(\theta)}(A,B)
				&=
				\Tr\Bigg[
				M_r
				\Bigg(
				-\frac1{r^2}\log C_r
				+
				\frac1r
				D\log(C_r)
				\big[
				(1-\theta)A^r\log A+\theta B^r\log B
				\big]
				\Bigg)
				\Bigg].
			\end{aligned}
			\]
		\end{proposition}
		
		\begin{proof}
			Write
			\[
			T_r^{(\theta)}(A,B)
			=
			\Tr\exp\left(\frac1r\log C_r\right).
			\]
			Let
			\[
			L_r=\frac1r\log C_r.
			\]
			By the trace derivative formula for the exponential,
			\[
			\frac{d}{dr}\Tr e^{L_r}
			=
			\Tr(e^{L_r}L_r').
			\]
			Moreover,
			\[
			C_r'
			=
			(1-\theta)A^r\log A+\theta B^r\log B,
			\]
			and hence
			\[
			L_r'
			=
			-\frac1{r^2}\log C_r
			+
			\frac1rD\log(C_r)[C_r'].
			\]
			Substituting \(e^{L_r}=M_r\) gives the desired formula.
		\end{proof}
		
		Taking \(r=1\) in the preceding derivative formula gives an entropy-type
		limit. Let
		\[
		S(X)=-\Tr X\log X
		\]
		denote the von Neumann entropy functional \cite{vonNeumann1955}.
		
		\begin{theorem}\label{thm:qjsd-limit}
			For \(A,B\in\Pn\),
			\[
			\lim_{\varepsilon\downarrow0}
			\frac{\Phi_{1+\varepsilon,1}^{(\theta)}(A,B)}{\varepsilon}
			=
			S((1-\theta)A+\theta B)
			-
			(1-\theta)S(A)
			-
			\theta S(B).
			\]
			In particular, when \(A\) and \(B\) are density matrices, the
			right-hand side is the weighted quantum Jensen--Shannon divergence
			\cite{LambertiQJSD}.
		\end{theorem}
		
		\begin{proof}
			Put
			\[
			C=(1-\theta)A+\theta B.
			\]
			At \(r=1\), Proposition~\ref{prop:diagonal-parameter-derivative} gives
			\[
			\left.\frac{\partial}{\partial r}T_r^{(\theta)}(A,B)\right|_{r=1}
			=
			-\Tr C\log C
			+
			\Tr\left[
			C\,D\log(C)
			\big[
			(1-\theta)A\log A+\theta B\log B
			\big]
			\right].
			\]
			For any Hermitian \(E\),
			\[
			\Tr\big(C\,D\log(C)[E]\big)=\Tr E.
			\]
			Indeed, this follows by differentiating
			\[
			\Tr(C+tE)=\Tr\exp(\log(C+tE))
			\]
			at \(t=0\). Hence
			\[
			\left.\frac{\partial}{\partial r}T_r^{(\theta)}(A,B)\right|_{r=1}
			=
			-\Tr C\log C
			+
			(1-\theta)\Tr A\log A
			+
			\theta\Tr B\log B.
			\]
			Using \(S(X)=-\Tr X\log X\), this is
			\[
			S(C)-(1-\theta)S(A)-\theta S(B),
			\]
			as claimed.
		\end{proof}
		
		We next turn to the logarithmic endpoint \(r=0\), which is not covered by
		the preceding positive-parameter derivative formula.
		
		\begin{theorem}\label{thm:local-log-euclidean}
			For \(A,B\in\Pn\),
			\[
			\lim_{\varepsilon\downarrow0}
			\frac{\Phi_{\varepsilon,0}^{(\theta)}(A,B)}{\varepsilon}
			=
			\frac{\theta(1-\theta)}2
			\Tr\left[
			M_0^{(\theta)}(A,B)(\log A-\log B)^2
			\right].
			\]
			In particular, this limit is nonnegative and vanishes if and only if
			\(A=B\).
		\end{theorem}
		
		\begin{proof}
			Set
			\[
			X=\log A,\qquad Y=\log B,
			\]
			and
			\[
			L=(1-\theta)X+\theta Y.
			\]
			Then
			\[
			M_0^{(\theta)}(A,B)=e^L.
			\]
			For \(\varepsilon>0\),
			\[
			(1-\theta)A^\varepsilon+\theta B^\varepsilon
			=
			(1-\theta)e^{\varepsilon X}+\theta e^{\varepsilon Y}.
			\]
			All \(O(\varepsilon^k)\) terms below are in operator norm. Since the
			logarithm is analytic in a neighbourhood of the identity and the
			exponential is analytic everywhere, the following expansions may be used
			inside the trace by finite-dimensional analytic functional calculus.
			Expanding at \(\varepsilon=0\), we get
			\[
			(1-\theta)e^{\varepsilon X}+\theta e^{\varepsilon Y}
			=
			I+\varepsilon L+\frac{\varepsilon^2}{2}
			\big((1-\theta)X^2+\theta Y^2\big)
			+
			O(\varepsilon^3).
			\]
			Therefore,
			\[
			\log\big((1-\theta)A^\varepsilon+\theta B^\varepsilon\big)
			=
			\varepsilon L
			+
			\frac{\varepsilon^2}{2}
			\left[
			(1-\theta)X^2+\theta Y^2-L^2
			\right]
			+
			O(\varepsilon^3).
			\]
			Dividing by \(\varepsilon\), we obtain
			\[
			\frac1\varepsilon
			\log\big((1-\theta)A^\varepsilon+\theta B^\varepsilon\big)
			=
			L
			+
			\frac{\varepsilon}{2}
			\left[
			(1-\theta)X^2+\theta Y^2-L^2
			\right]
			+
			O(\varepsilon^2).
			\]
			Since
			\[
			(1-\theta)X^2+\theta Y^2-L^2
			=
			\theta(1-\theta)(X-Y)^2,
			\]
			where \((X-Y)^2=X^2-XY-YX+Y^2\), no commutativity of \(X\) and \(Y\) is
			being used. Hence
			\[
			M_\varepsilon^{(\theta)}(A,B)
			=
			\exp\left[
			L+\frac{\varepsilon}{2}\theta(1-\theta)(X-Y)^2
			+
			O(\varepsilon^2)
			\right].
			\]
			Taking traces and differentiating the trace exponential gives
			\[
			T_\varepsilon^{(\theta)}(A,B)
			=
			T_0^{(\theta)}(A,B)
			+
			\frac{\varepsilon}{2}\theta(1-\theta)
			\Tr\left[e^L(X-Y)^2\right]
			+
			O(\varepsilon^2).
			\]
			Hence
			\[
			\lim_{\varepsilon\downarrow0}
			\frac{T_\varepsilon^{(\theta)}(A,B)-T_0^{(\theta)}(A,B)}
			{\varepsilon}
			=
			\frac{\theta(1-\theta)}2
			\Tr\left[
			M_0^{(\theta)}(A,B)(\log A-\log B)^2
			\right].
			\]
			This is exactly the asserted limit.
			
			To prove the equality condition, set
			\[
			P=M_0^{(\theta)}(A,B),
			\qquad
			Z=\log A-\log B.
			\]
			Then \(P>0\), \(Z=Z^*\), and
			\[
			\Tr(PZ^2)
			=
			\Tr(ZPZ)
			=
			\Tr\big((P^{1/2}Z)^*(P^{1/2}Z)\big)\ge0.
			\]
			If the limit is zero, then
			\[
			\Tr\big((P^{1/2}Z)^*(P^{1/2}Z)\big)=0,
			\]
			and hence
			\[
			P^{1/2}Z=0.
			\]
			Since \(P\) is invertible, \(Z=0\). Thus \(\log A=\log B\), and hence
			\(A=B\). The converse is immediate.
		\end{proof}

		\subsection{Two structural consequences}
		
		We close this section with two structural consequences of the divergence
		proof: monotonicity with respect to the parameters and the explicit local
		quadratic form induced on the diagonal.
		
		The first consequence follows directly from the nonnegativity and equality
		condition in Proposition~\ref{prop:nonneg}.
		
		\begin{corollary}\label{cor:parameter-monotonicity}
			Let \(0<\theta<1\) and let \(A,B\in\Pn\). The map
			\[
			r\longmapsto T_r^{(\theta)}(A,B)
			\]
			is increasing on \([0,\infty)\). Moreover, if \(A\ne B\), then it is
			strictly increasing.
			
			Consequently, if
			\[
			p_1>p_2>q\ge0,
			\]
			then
			\[
			\Phi_{p_1,q}^{(\theta)}(A,B)
			\ge
			\Phi_{p_2,q}^{(\theta)}(A,B),
			\]
			and if
			\[
			p>q_1>q_2\ge0,
			\]
			then
			\[
			\Phi_{p,q_2}^{(\theta)}(A,B)
			\ge
			\Phi_{p,q_1}^{(\theta)}(A,B).
			\]
			In both inequalities, equality holds if and only if \(A=B\).
		\end{corollary}
		
		\begin{proof}
			For \(s>r\ge0\), we have
			\[
			T_s^{(\theta)}(A,B)-T_r^{(\theta)}(A,B)
			=
			\Phi_{s,r}^{(\theta)}(A,B).
			\]
			Proposition~\ref{prop:nonneg} gives
			\[
			\Phi_{s,r}^{(\theta)}(A,B)\ge0,
			\]
			with equality if and only if \(A=B\). This proves the monotonicity and
			the strictness statement. The two displayed inequalities follow from
			\[
			\Phi_{p_1,q}^{(\theta)}-\Phi_{p_2,q}^{(\theta)}
			=
			\Phi_{p_1,p_2}^{(\theta)}
			\]
			and
			\[
			\Phi_{p,q_2}^{(\theta)}-\Phi_{p,q_1}^{(\theta)}
			=
			\Phi_{q_1,q_2}^{(\theta)}.
			\]
		\end{proof}
		
		The second consequence records the local quadratic form induced on the
		diagonal. Although it follows from the Hessian computation above, we state
		it separately for later use.
		
		\begin{proposition}\label{prop:local-metric}
			Let \(0<\theta<1\) and \(p>q\ge0\). Define
			\[
			g_{A,\theta}^{p,q}(H,H)
			=
			D_2^2\Phi_{p,q}^{(\theta)}(A,A)[H,H].
			\]
			If \(A=\diag(a_1,\ldots,a_n)\) and \(H=(h_{ij})\in\Hn\), then
			\[
			g_{A,\theta}^{p,q}(H,H)
			=
			4\theta(1-\theta)
			\left[
			\sum_i\frac{p-q}{4a_i}|h_{ii}|^2
			+
			\sum_{i<j}
			\big(\Gamma_p(a_i,a_j)-\Gamma_q(a_i,a_j)\big)|h_{ij}|^2
			\right],
			\]
			where \(\Gamma_0\) is understood by the convention
			\[
			\Gamma_0(a,b)
			=
			\frac{\log(b/a)}{2(a-b)}
			\quad (a\ne b),
			\qquad
			\Gamma_0(a,a)=-\frac1{2a}.
			\]
			In particular, \(g_{A,\theta}^{p,q}\) is positive definite on \(\Hn\).
		\end{proposition}
		
		\begin{proof}
			This is Proposition~\ref{prop:hessian-Tr} applied to
			\(T_p^{(\theta)}-T_q^{(\theta)}\), together with
			Theorem~\ref{thm:hessian-positive}.
		\end{proof}
		\section{Joint convexity and data processing}\label{sec:convexity-dpi}
		
		We begin with the convexity properties of the trace functional
		\(T_r^{(\theta)}\). For \(r>0\), the weights can be absorbed by the
		invertible linear change of variables
		\[
		X=(1-\theta)^{1/r}A,\qquad Y=\theta^{1/r}B,
		\]
		so that
		\[
		T_r^{(\theta)}(A,B)
		=
		\Tr\big(X^r+Y^r\big)^{1/r}.
		\]
		Hence the Minkowski-type trace theorem of Carlen and Lieb
		\cite{CarlenLiebII} implies that \(T_r^{(\theta)}\) is jointly concave for
		\(0<r\le1\) and jointly convex for \(1\le r\le2\). At the logarithmic
		endpoint,
		\[
		T_0^{(\theta)}(A,B)
		=
		\lim_{r\downarrow0}T_r^{(\theta)}(A,B)
		=
		\Tr\exp\big((1-\theta)\log A+\theta\log B\big).
		\]
		Since a finite pointwise limit of jointly concave functions is jointly
		concave, \(T_0^{(\theta)}\) is jointly concave.
		\begin{theorem}\label{thm:joint-convexity}
			Let $0<\theta<1$ and
			\[
			0\le q\le1\le p\le2,
			\qquad
			p>q.
			\]
			Then $\Phi_{p,q}^{(\theta)}$ is jointly convex on $\Pn\times\Pn$.
		\end{theorem}
		
		\begin{proof}
			Let $0\le\lambda\le1$ and put
			\[
			A_\lambda=\lambda A_1+(1-\lambda)A_2,
			\qquad
			B_\lambda=\lambda B_1+(1-\lambda)B_2.
			\]
			Since $1\le p\le2$, $T_p^{(\theta)}$ is jointly convex. Hence
			\[
			T_p^{(\theta)}(A_\lambda,B_\lambda)
			\le
			\lambda T_p^{(\theta)}(A_1,B_1)
			+
			(1-\lambda)T_p^{(\theta)}(A_2,B_2).
			\]
			Since $0\le q\le1$, $T_q^{(\theta)}$ is jointly concave. Hence
			\[
			T_q^{(\theta)}(A_\lambda,B_\lambda)
			\ge
			\lambda T_q^{(\theta)}(A_1,B_1)
			+
			(1-\lambda)T_q^{(\theta)}(A_2,B_2).
			\]
			Subtracting the second inequality from the first gives
			\[
			\Phi_{p,q}^{(\theta)}(A_\lambda,B_\lambda)
			\le
			\lambda\Phi_{p,q}^{(\theta)}(A_1,B_1)
			+
			(1-\lambda)\Phi_{p,q}^{(\theta)}(A_2,B_2).
			\]
		\end{proof}
		We next prove the data processing inequality. The argument follows the
		standard convexity--dilation--partial-trace route. We use the finite-dimensional Stinespring representation for quantum
		channels: every completely positive trace-preserving map can be written in
		the form
		\[
		\Lambda(X)=\operatorname{Tr}_{\mathcal E}(VXV^*),
		\]
		where \(V\) is an isometry into a system-environment space and
		\(\operatorname{Tr}_{\mathcal E}\) denotes the partial trace over the
		environment space \(\mathcal E\) \cite[Corollary~2.27]{Watrous}. Even when \(X>0\), the operator \(VXV^*\) is typically
		singular on the larger space, since it vanishes on the orthogonal
		complement of the range of \(V\). Thus the logarithmic endpoint \(r=0\)
		forces us to specify a boundary convention before applying the dilation
		argument.
		
		\begin{lemma}\label{lem:regularized-boundary}
			Let \(X,Y\ge0\). For \(r>0\), define
			\[
			T_r^{(\theta)}(X,Y)
			=
			\Tr\big((1-\theta)X^r+\theta Y^r\big)^{1/r}
			\]
			by continuous functional calculus. For \(r=0\), define
			\[
			T_0^{(\theta)}(X,Y)
			=
			\lim_{\varepsilon\downarrow0}
			\Tr\exp\big((1-\theta)\log(X+\varepsilon I)
			+
			\theta\log(Y+\varepsilon I)\big).
			\]
			Then the limit exists. With these conventions, \(T_r^{(\theta)}\), and
			hence \(\Phi_{p,q}^{(\theta)}=T_p^{(\theta)}-T_q^{(\theta)}\), is
			compatible with unitary conjugation, direct sums, and tensoring by a
			normalized identity whenever boundary values occur.
		\end{lemma}
		
		\begin{proof}
			Only the logarithmic endpoint requires proof. Put
			\[
			F_\varepsilon(X,Y)
			=
			\Tr\exp\big((1-\theta)\log(X+\varepsilon I)
			+
			\theta\log(Y+\varepsilon I)\big).
			\]
			If \(0<\varepsilon<\delta\), then
			\[
			X+\varepsilon I\le X+\delta I,
			\qquad
			Y+\varepsilon I\le Y+\delta I.
			\]
			Since the logarithm is operator monotone on \((0,\infty)\), we have
			\[
			(1-\theta)\log(X+\varepsilon I)+\theta\log(Y+\varepsilon I)
			\le
			(1-\theta)\log(X+\delta I)+\theta\log(Y+\delta I).
			\]
			If \(H\le K\), then Weyl monotonicity \cite[p.~63]{BhatiaMA} gives
			\(\lambda_j(H)\le\lambda_j(K)\) for every \(j\). Since the scalar
			exponential is increasing, this implies
			\(\Tr e^H\le\Tr e^K\). Hence
			\[
			F_\varepsilon(X,Y)\le F_\delta(X,Y).
			\]
			Thus \(F_\varepsilon(X,Y)\) is monotone increasing as \(\varepsilon\)
			increases, and the limit as \(\varepsilon\downarrow0\) exists in
			\([0,\infty)\).
			
			Unitary covariance is immediate from functional calculus. Direct sums
			follow because, for every \(\varepsilon>0\), the logarithm and exponential
			preserve block diagonal structure, and the trace is additive over direct
			sums.
			
			It remains to record the normalized tensor identity at \(r=0\). For
			\(m\in\mathbb N\),
			\[
			X\otimes\frac{I_m}{m}+\varepsilon I_{nm}
			=
			(X+m\varepsilon I_n)\otimes\frac{I_m}{m},
			\]
			and similarly for \(Y\). Hence
			\[
			F_\varepsilon\left(
			X\otimes\frac{I_m}{m},
			Y\otimes\frac{I_m}{m}
			\right)
			=
			F_{m\varepsilon}(X,Y).
			\]
			Letting \(\varepsilon\downarrow0\) gives
			\[
			T_0^{(\theta)}\left(
			X\otimes\frac{I_m}{m},
			Y\otimes\frac{I_m}{m}
			\right)
			=
			T_0^{(\theta)}(X,Y).
			\]
			The corresponding assertions for \(r>0\) follow directly from continuous
			functional calculus:
			\[
			M_r^{(\theta)}\left(
			X\otimes\frac{I_m}{m},
			Y\otimes\frac{I_m}{m}
			\right)
			=
			M_r^{(\theta)}(X,Y)\otimes\frac{I_m}{m}.
			\]
			This proves all stated compatibility properties.
		\end{proof}
		
		\begin{lemma}\label{lem:unitary-tensor}
			Let \(0<\theta<1\) and \(p>q\ge0\). The boundary-extended functional
			\(\Phi_{p,q}^{(\theta)}\) is unitarily invariant and satisfies
			\[
			\Phi_{p,q}^{(\theta)}
			\left(A\otimes\frac{I_m}{m},B\otimes\frac{I_m}{m}\right)
			=
			\Phi_{p,q}^{(\theta)}(A,B)
			\]
			for all positive semidefinite \(A,B\) for which the two sides are
			interpreted by Lemma~\ref{lem:regularized-boundary}.
		\end{lemma}
		
		\begin{proof}
			This is the unitary and normalized tensor compatibility in
			Lemma~\ref{lem:regularized-boundary}, applied separately to \(T_p\) and
			\(T_q\).
		\end{proof}
		
		\begin{lemma}\label{lem:partial-trace-monotonicity}
			Assume
			\[
			0\le q\le1\le p\le2,
			\qquad p>q.
			\]
			Then, for positive definite \(X,Y\in M_n(\mathbb C)\otimes M_m(\mathbb C)\),
			\[
			\Phi_{p,q}^{(\theta)}(\Tr_2X,\Tr_2Y)
			\le
			\Phi_{p,q}^{(\theta)}(X,Y).
			\]
		\end{lemma}
		
		\begin{proof}
			Let \(\{U_\alpha\}_{\alpha=1}^{m^2}\) be a unitary error basis on
			\(\mathbb C^m\). The partial twirling identity gives
			\[
			\frac1{m^2}
			\sum_{\alpha=1}^{m^2}
			(I\otimes U_\alpha)X(I\otimes U_\alpha^*)
			=
			(\Tr_2X)\otimes\frac{I_m}{m},
			\]
			and the same identity holds for \(Y\). By joint convexity and unitary
			invariance,
			\[
			\begin{aligned}
				&\Phi_{p,q}^{(\theta)}
				\left(
				(\Tr_2X)\otimes\frac{I_m}{m},
				(\Tr_2Y)\otimes\frac{I_m}{m}
				\right) \\
				&\quad\le
				\frac1{m^2}
				\sum_{\alpha=1}^{m^2}
				\Phi_{p,q}^{(\theta)}
				\left(
				(I\otimes U_\alpha)X(I\otimes U_\alpha^*),
				(I\otimes U_\alpha)Y(I\otimes U_\alpha^*)
				\right) \\
				&\quad=
				\Phi_{p,q}^{(\theta)}(X,Y).
			\end{aligned}
			\]
			Using Lemma~\ref{lem:unitary-tensor} on the left-hand side gives the
			claimed partial trace inequality.
		\end{proof}
		
		\begin{theorem}\label{thm:dpi}
			Let \(0<\theta<1\) and
			\[
			0\le q\le1\le p\le2,
			\qquad
			p>q.
			\]
			Then \(\Phi_{p,q}^{(\theta)}\) satisfies the data processing inequality
			\[
			\Phi_{p,q}^{(\theta)}(\Lambda(A),\Lambda(B))
			\le
			\Phi_{p,q}^{(\theta)}(A,B)
			\]
			for all \(A,B\in\Pn\) and every completely positive trace-preserving map
			\(\Lambda:M_n(\mathbb C)\to M_m(\mathbb C)\), with the same formula used
			on the appropriate input and output matrix dimensions. If \(\Lambda(A)\)
			or \(\Lambda(B)\) is singular, then the left-hand side is interpreted
			through the regularized boundary convention of
			Lemma~\ref{lem:regularized-boundary}.
		\end{theorem}
		
		\begin{proof}
			Let \(\mathcal H\) and \(\mathcal K\) be the input and output spaces of
			\(\Lambda\), respectively. By the finite-dimensional Stinespring
			representation for quantum channels \cite[Corollary~2.27]{Watrous},
			there exist an environment space \(\mathcal E\) and an isometry
			\[
			V:\mathcal H\to\mathcal K\otimes\mathcal E
			\]
			such that
			\[
			\Lambda(X)=\Tr_{\mathcal E}(VXV^*)
			\]
			for every \(X\in M_n(\mathbb C)\), where \(\Tr_{\mathcal E}\) denotes
			the partial trace over \(\mathcal E\). Put \(d_E=\dim\mathcal E\).
			
			The operators \(VAV^*\) and \(VBV^*\) are positive semidefinite on
			\(\mathcal K\otimes\mathcal E\), but they may be singular. For
			\(\varepsilon>0\), set
			\[
			X_\varepsilon
			=
			VAV^*+\varepsilon I_{\mathcal K\otimes\mathcal E},
			\qquad
			Y_\varepsilon
			=
			VBV^*+\varepsilon I_{\mathcal K\otimes\mathcal E}.
			\]
			Then \(X_\varepsilon,Y_\varepsilon\) are positive definite. By
			Lemma~\ref{lem:partial-trace-monotonicity},
			\[
			\Phi_{p,q}^{(\theta)}
			\big(
			\Tr_{\mathcal E}X_\varepsilon,
			\Tr_{\mathcal E}Y_\varepsilon
			\big)
			\le
			\Phi_{p,q}^{(\theta)}(X_\varepsilon,Y_\varepsilon).
			\]
			Since
			\[
			\Tr_{\mathcal E}X_\varepsilon
			=
			\Lambda(A)+\varepsilon d_E I_{\mathcal K},
			\qquad
			\Tr_{\mathcal E}Y_\varepsilon
			=
			\Lambda(B)+\varepsilon d_E I_{\mathcal K},
			\]
			the left-hand side converges, in the sense of
			Lemma~\ref{lem:regularized-boundary}, to
			\[
			\Phi_{p,q}^{(\theta)}(\Lambda(A),\Lambda(B)).
			\]
			
			On the other hand, with respect to the orthogonal decomposition
			\[
			\mathcal K\otimes\mathcal E
			=
			V\mathcal H\oplus(V\mathcal H)^\perp,
			\]
			we have
			\[
			X_\varepsilon
			=
			V(A+\varepsilon I_{\mathcal H})V^*
			\oplus
			\varepsilon I_{(V\mathcal H)^\perp},
			\]
			and similarly
			\[
			Y_\varepsilon
			=
			V(B+\varepsilon I_{\mathcal H})V^*
			\oplus
			\varepsilon I_{(V\mathcal H)^\perp}.
			\]
			The complementary block contributes zero to the trace gap, because the
			two entries are identical there. Hence, by unitary invariance and
			additivity over direct sums,
			\[
			\Phi_{p,q}^{(\theta)}(X_\varepsilon,Y_\varepsilon)
			=
			\Phi_{p,q}^{(\theta)}(A+\varepsilon I_{\mathcal H},
			B+\varepsilon I_{\mathcal H}).
			\]
			Letting \(\varepsilon\downarrow0\) and using continuity on the positive
			definite cone gives
			\[
			\Phi_{p,q}^{(\theta)}(\Lambda(A),\Lambda(B))
			\le
			\Phi_{p,q}^{(\theta)}(A,B).
			\]
			This proves the data processing inequality.
		\end{proof}
		
		\begin{corollary}\label{cor:metric-monotonicity}
			Assume
			\[
			0\le q\le1\le p\le2,
			\qquad p>q.
			\]
			Let $\Lambda$ be a completely positive trace-preserving map and let
			$A\in\Pn$ be such that $\Lambda(A)$ is positive definite. Then, for every
			$H\in\Hn$,
			\[
			g_{\Lambda(A),\theta}^{p,q}(\Lambda(H),\Lambda(H))
			\le
			g_{A,\theta}^{p,q}(H,H).
			\]
		\end{corollary}
		
		\begin{proof}
			For \(|s|\) sufficiently small, apply Theorem~\ref{thm:dpi} to the pair
			\((A,A+sH)\). Both sides agree to order zero on the diagonal and have
			vanishing first derivative by Proposition~\ref{prop:first-derivative}.
			Comparing the second-order terms at \(s=0\) gives the claim. The
			assumption \(\Lambda(A)>0\) ensures that the left-hand local quadratic
			form is evaluated in the interior of the positive cone.
		\end{proof}
		
		\section{Scalar obstructions and counterexamples outside the convexity range}\label{sec:sharpness}
		
		In this section we sharpen the discussion of the parameter range in
		Theorem~\ref{thm:joint-convexity}. The results below show that the
		convexity and data-processing range cannot be extended to all
		$p>q\ge0$. In the region \(0<q<p<1\), we give an exact one-dimensional scalar
		criterion. This should not be confused with a closed-form classification
		of matrix joint convexity in that region.
		
		\subsection{Weighted scalar reduction}
		
		For $0<\theta<1$ and $r>0$, define the scalar weighted power mean
		\[
		m_{r,\theta}(x)
		=
		\big((1-\theta)+\theta x^r\big)^{1/r},
		\qquad x>0,
		\]
		and define its logarithmic endpoint by
		\[
		m_{0,\theta}(x)=x^\theta.
		\]
		The scalar restriction of $\Phi_{p,q}^{(\theta)}$ along $A=1$, $B=x$ is
		\[
		\phi_{p,q,\theta}(x)
		=
		m_{p,\theta}(x)-m_{q,\theta}(x).
		\]
		If $\Phi_{p,q}^{(\theta)}$ is jointly convex, then
		$\phi_{p,q,\theta}$ must be convex on $(0,\infty)$.
		
		A direct calculation gives, for $r>0$,
		\[
		m_{r,\theta}''(x)
		=
		(r-1)\theta(1-\theta)x^{r-2}
		\big((1-\theta)+\theta x^r\big)^{1/r-2}.
		\]
		For $r=0$,
		\[
		m_{0,\theta}''(x)
		=
		-\theta(1-\theta)x^{\theta-2}.
		\]
		
		\subsection{Failure for every weight when \(1<q<p\)}
		
		\begin{theorem}\label{thm:failure-q-greater-one}
			Let $0<\theta<1$ and $1<q<p$. Then
			$\Phi_{p,q}^{(\theta)}$ is not jointly convex.
		\end{theorem}
		
		\begin{proof}
			Using the formula above, as $x\downarrow0$ we have
			\[
			m_{p,\theta}''(x)-m_{q,\theta}''(x)
			=
			\theta(1-\theta)
			\left[
			(p-1)(1-\theta)^{1/p-2}x^{p-2}
			-
			(q-1)(1-\theta)^{1/q-2}x^{q-2}
			\right](1+o(1)).
			\]
			Since $p>q$, we have
			\[
			x^{p-2}=o(x^{q-2})
			\qquad (x\downarrow0).
			\]
			The second term is negative and dominates. Hence
			\[
			\phi_{p,q,\theta}''(x)<0
			\]
			for all sufficiently small $x>0$. Therefore the scalar restriction is not
			convex, so $\Phi_{p,q}^{(\theta)}$ cannot be jointly convex.
		\end{proof}
		
		\subsection{The region \(0<q<p<1\): a scalar criterion and an explicit example}
		
		For $0<q<p<1$, both $m_{p,\theta}$ and $m_{q,\theta}$ are concave. Thus
		the sign of
		\[
		\phi_{p,q,\theta}''(x)
		=
		m_{p,\theta}''(x)-m_{q,\theta}''(x)
		\]
		is more delicate. Define
		\[
		K_{r,\theta}(x)
		=
		(1-r)x^{r-2}
		\big((1-\theta)+\theta x^r\big)^{1/r-2},
		\qquad 0<r<1.
		\]
		Then
		\[
		m_{r,\theta}''(x)=-\theta(1-\theta)K_{r,\theta}(x),
		\]
		and hence
		\[
		\phi_{p,q,\theta}''(x)
		=
		\theta(1-\theta)\big(K_{q,\theta}(x)-K_{p,\theta}(x)\big).
		\]
		
		Put $x=e^t$ and define
		\[
		\Delta_{\theta,p,q}(t)
		=
		\log\frac{K_{p,\theta}(e^t)}{K_{q,\theta}(e^t)}.
		\]
		Explicitly,
		\[
		\begin{aligned}
			\Delta_{\theta,p,q}(t)
			&=
			\log\frac{1-p}{1-q}
			+(p-q)t \\
			&\quad+
			\left(\frac1p-2\right)
			\log\big((1-\theta)+\theta e^{pt}\big)
			-
			\left(\frac1q-2\right)
			\log\big((1-\theta)+\theta e^{qt}\big).
		\end{aligned}
		\]
		Since $p>q$, one has
		\[
		\Delta_{\theta,p,q}(t)\to-\infty
		\qquad (t\to\pm\infty).
		\]
		Therefore the following number is finite and attained:
		\[
		\mathcal S_\theta(p,q)
		=
		\sup_{t\in\mathbb R}\Delta_{\theta,p,q}(t).
		\]
		
		\begin{theorem}\label{thm:scalar-criterion-0qp1}
			Let $0<\theta<1$ and $0<q<p<1$. Then the scalar restriction
			\[
			x\longmapsto \phi_{p,q,\theta}(x)
			\]
			is convex on $(0,\infty)$ if and only if
			\[
			\mathcal S_\theta(p,q)\le0.
			\]
			Equivalently, scalar convexity fails if and only if
			\[
			\mathcal S_\theta(p,q)>0.
			\]
			Moreover, any interior maximizer $t$ of $\Delta_{\theta,p,q}$ satisfies
			\[
			(p-q)
			+
			(1-2p)\frac{\theta e^{pt}}{(1-\theta)+\theta e^{pt}}
			-
			(1-2q)\frac{\theta e^{qt}}{(1-\theta)+\theta e^{qt}}
			=
			0.
			\]
		\end{theorem}
		
		\begin{proof}
			For $0<q<p<1$,
			\[
			\phi_{p,q,\theta}''(x)
			=
			\theta(1-\theta)\big(K_{q,\theta}(x)-K_{p,\theta}(x)\big).
			\]
			Since $\theta(1-\theta)>0$, the scalar function
			$\phi_{p,q,\theta}$ is convex if and only if
			\[
			K_{q,\theta}(x)\ge K_{p,\theta}(x)
			\qquad\text{for all }x>0.
			\]
			This is equivalent to
			\[
			\frac{K_{p,\theta}(x)}{K_{q,\theta}(x)}\le1
			\qquad\text{for all }x>0,
			\]
			or, with $x=e^t$,
			\[
			\Delta_{\theta,p,q}(t)\le0
			\qquad\text{for all }t\in\mathbb R.
			\]
			Thus scalar convexity is equivalent to
			$\mathcal S_\theta(p,q)\le0$.
			
			The critical point equation is obtained by differentiating the displayed
			formula for $\Delta_{\theta,p,q}(t)$.
		\end{proof}
		
		Theorem~\ref{thm:scalar-criterion-0qp1} is a scalar result: it reduces
		convexity of the one-dimensional restriction to the sign of
		\(\mathcal S_\theta(p,q)\). It does not give a closed-form classification
		or a matrix-level classification of joint convexity in the region
		\(0<q<p<1\).
		
		As a concrete illustration, the criterion detects explicit failures of
		scalar convexity. For the symmetric weight \(\theta=1/2\), take
		\[
		p=\frac15,\qquad q=\frac1{10},
		\qquad x=10^{-10}.
		\]
		For
		\[
		f_r(x)=\left(\frac{1+x^r}{2}\right)^{1/r},
		\]
		we have
		\[
		f_r''(x)
		=
		\frac{
			(r-1)x^{r-2}
			\left(\dfrac{1+x^r}{2}\right)^{1/r}
		}{
			(1+x^r)^2
		}.
		\]
		Thus
		\[
		f_{1/5}''(10^{-10})
		=
		-\frac45\,10^{18}
		\frac{(101/200)^5}{(101/100)^2}
		=
		-25757525000000000,
		\]
		and
		\[
		f_{1/10}''(10^{-10})
		=
		-\frac9{10}\,10^{19}
		\frac{(11/20)^{10}}{(11/10)^2}
		=
		-18840136025390625.
		\]
		Consequently,
		\[
		(f_{1/5}-f_{1/10})''(10^{-10})
		=
		-6917388974609375<0.
		\]
		Hence the scalar restriction of \(\Phi_{1/5,1/10}^{(1/2)}\) is not
		convex, and therefore \(\Phi_{1/5,1/10}^{(1/2)}\) is not jointly convex.
		
		\subsection{The scalar endpoint \(q=0\)}
		
		The scalar logarithmic endpoint admits a closed-form classification.
		
		\begin{theorem}\label{thm:q0-scalar-classification}
			Let \(0<\theta<1\) and \(p>0\). The scalar function
			\[
			x\longmapsto m_{p,\theta}(x)-m_{0,\theta}(x)
			=
			\big((1-\theta)+\theta x^p\big)^{1/p}-x^\theta
			\]
			is convex on \((0,\infty)\) if and only if
			\[
			p\ge\max\{\theta,1-\theta\}.
			\]
			Equivalently, it is nonconvex if and only if
			\[
			0<p<\max\{\theta,1-\theta\}.
			\]
			Consequently, if
			\[
			0<p<\max\{\theta,1-\theta\},
			\]
			then \(\Phi_{p,0}^{(\theta)}\) is not jointly convex.
		\end{theorem}
		
		\begin{proof}
			Set
			\[
			a=1-\theta,\qquad b=\theta.
			\]
			
			First assume \(p\ge1\). Then
			\[
			m_{p,\theta}''(x)
			=
			(p-1)\theta(1-\theta)x^{p-2}
			\big((1-\theta)+\theta x^p\big)^{1/p-2}\ge0.
			\]
			Moreover,
			\[
			\frac{d^2}{dx^2}(-x^\theta)
			=
			\theta(1-\theta)x^{\theta-2}\ge0.
			\]
			Hence \(m_{p,\theta}-m_{0,\theta}\) is convex for every \(p\ge1\).
			Since \(\max\{\theta,1-\theta\}<1\), it remains to analyze \(0<p<1\).
			
			For \(0<p<1\), define
			\[
			K_{p,\theta}(x)
			=
			(1-p)x^{p-2}(a+bx^p)^{1/p-2}.
			\]
			Since
			\[
			m_{p,\theta}''(x)=-abK_{p,\theta}(x),
			\qquad
			m_{0,\theta}''(x)=-abx^{b-2},
			\]
			the function \(m_{p,\theta}-m_{0,\theta}\) is convex if and only if
			\[
			K_{p,\theta}(x)\le x^{b-2}
			\qquad\text{for all }x>0.
			\]
			Equivalently,
			\[
			R(x):=
			\frac{K_{p,\theta}(x)}{x^{b-2}}
			=
			(1-p)x^{p-b}(a+bx^p)^{1/p-2}
			\le1
			\qquad\text{for all }x>0.
			\]
			
			Suppose first that \(p<b=\theta\). As \(x\downarrow0\),
			\[
			R(x)
			\sim
			(1-p)a^{1/p-2}x^{p-b}\to\infty.
			\]
			Thus convexity fails.
			
			If \(p<a=1-\theta\), then as \(x\to\infty\),
			\[
			R(x)
			\sim
			(1-p)b^{1/p-2}x^{a-p}\to\infty.
			\]
			Thus convexity also fails.
			
			It remains to prove convexity when \(0<p<1\) and
			\[
			p\ge\max\{a,b\}.
			\]
			In this case \(p\ge1/2\), so \(1/p-2\le0\). If \(0<x\le1\), then
			\[
			a+bx^p\ge a,
			\]
			and \(p\ge b\). Therefore
			\[
			R(x)
			\le
			(1-p)a^{1/p-2}
			=
			(1-p)a^{-(2p-1)/p}.
			\]
			Since \(p\ge b\), we have \(1-p\le a\). Since
			\[
			0\le\frac{2p-1}{p}\le1,
			\]
			we have
			\[
			a^{(2p-1)/p}\ge a\ge1-p.
			\]
			Hence \(R(x)\le1\) for \(0<x\le1\).
			
			If \(x\ge1\), then
			\[
			a+bx^p\ge bx^p,
			\]
			and \(p\ge a\). Thus
			\[
			R(x)
			\le
			(1-p)b^{1/p-2}x^{a-p}
			\le
			(1-p)b^{-(2p-1)/p}.
			\]
			Since \(p\ge a\), we have \(1-p\le b\), and again
			\[
			b^{(2p-1)/p}\ge b\ge1-p.
			\]
			Hence \(R(x)\le1\) for \(x\ge1\).
			
			Therefore \(R(x)\le1\) for all \(x>0\), which proves scalar convexity
			when \(p\ge\max\{\theta,1-\theta\}\). Combining this with the two
			asymptotic obstructions proves the equivalence.
		\end{proof}

		\begin{corollary}\label{cor:small-q-failure}
			Let $0<\theta<1$ and
			\[
			0<p<\max\{\theta,1-\theta\}.
			\]
			Then there exists $\varepsilon>0$ such that, for every
			\[
			0<q<\min\{p,\varepsilon\},
			\]
			the functional $\Phi_{p,q}^{(\theta)}$ is not jointly convex.
		\end{corollary}
		
		\begin{proof}
			By Theorem~\ref{thm:q0-scalar-classification}, there exists
			$x_0>0$ such that
			\[
			\frac{d^2}{dx^2}
			\left[
			m_{p,\theta}(x)-m_{0,\theta}(x)
			\right]_{x=x_0}
			<0.
			\]
			For fixed $x_0$, $m_{q,\theta}''(x_0)\to m_{0,\theta}''(x_0)$ as
			$q\downarrow0$. Hence, for all sufficiently small $q>0$,
			\[
			\frac{d^2}{dx^2}
			\left[
			m_{p,\theta}(x)-m_{q,\theta}(x)
			\right]_{x=x_0}
			<0.
			\]
			Thus the scalar restriction is nonconvex, so
			$\Phi_{p,q}^{(\theta)}$ is not jointly convex.
		\end{proof}

		\subsection{Failure for \(p>2\) already when \(q=1\)}
		
		\begin{theorem}\label{thm:p-greater-two-failure}
			Let \(0<\theta<1\) and \(p>2\). Then the upper trace functional
			\[
			T_p^{(\theta)}(A,B)
			=
			\Tr\big((1-\theta)A^p+\theta B^p\big)^{1/p}
			\]
			is not jointly convex in general. Consequently,
			\[
			\Phi_{p,1}^{(\theta)}=T_p^{(\theta)}-T_1^{(\theta)}
			\]
			is not jointly convex in general.
		\end{theorem}
		
		\begin{proof}
			Since
			\[
			T_1^{(\theta)}(A,B)
			=
			(1-\theta)\Tr A+\theta\Tr B
			\]
			is affine, joint convexity of \(\Phi_{p,1}^{(\theta)}\) would imply joint
			convexity of \(T_p^{(\theta)}\).
			
			The functional \(T_p^{(\theta)}\) is obtained from the Carlen--Lieb trace
			functional
			\[
			(X,Y)\longmapsto \Tr(X^p+Y^p)^{1/p}
			\]
			by the invertible linear rescaling
			\[
			X=(1-\theta)^{1/p}A,\qquad Y=\theta^{1/p}B.
			\]
			Carlen and Lieb showed that the corresponding trace functional fails
			convexity for \(p>2\) in general~\cite{CarlenLiebI}. Hence
			\(T_p^{(\theta)}\), and therefore \(\Phi_{p,1}^{(\theta)}\), cannot be
			jointly convex in general.
		\end{proof}
		
		\begin{remark}
			For completeness, we record a concrete numerical counterexample to Jensen
			convexity in the symmetric case \(\theta=1/2\), \(p=3\), \(q=1\). Let
			\[
			A_1=
			\begin{pmatrix}
				1.51689413&0.73376717\\
				0.73376717&0.48310587
			\end{pmatrix},
			\quad
			A_2=
			\begin{pmatrix}
				1.24967062&0.39692589\\
				0.39692589&0.75032938
			\end{pmatrix},
			\]
			and
			\[
			B_1=
			\begin{pmatrix}
				1.58191486&0.74863040\\
				0.74863040&0.41808514
			\end{pmatrix},
			\quad
			B_2=
			\begin{pmatrix}
				1.54844564&0.11035089\\
				0.11035089&0.45155436
			\end{pmatrix}.
			\]
			All four matrices are positive definite. A direct calculation gives
			\[
			\Phi_{3,1}^{(1/2)}
			\left(
			\frac{A_1+A_2}{2},
			\frac{B_1+B_2}{2}
			\right)
			\approx 0.0846130853,
			\]
			whereas
			\[
			\frac12\Phi_{3,1}^{(1/2)}(A_1,B_1)
			+
			\frac12\Phi_{3,1}^{(1/2)}(A_2,B_2)
			\approx 0.0837764079.
			\]
			Thus Jensen convexity fails:
			\[
			\Phi_{3,1}^{(1/2)}
			\left(
			\frac{A_1+A_2}{2},
			\frac{B_1+B_2}{2}
			\right)
			>
			\frac12\Phi_{3,1}^{(1/2)}(A_1,B_1)
			+
			\frac12\Phi_{3,1}^{(1/2)}(A_2,B_2).
			\]
		\end{remark}

		\subsection{From convexity failure to data-processing failure}
		The preceding failures of joint convexity also imply failure of the data
		processing inequality.
		
		\begin{lemma}\label{lem:flag-failure}
			Suppose \(\Phi_{p,q}^{(\theta)}\) fails joint convexity. Then
			\(\Phi_{p,q}^{(\theta)}\) fails the data processing inequality for a
			partial-trace channel.
		\end{lemma}
		
		\begin{proof}
			If joint convexity fails, then there exist positive definite matrices
			\(A_1,A_2,B_1,B_2\) and a number \(0<\lambda<1\) such that
			\[
			\Phi_{p,q}^{(\theta)}
			(\lambda A_1+(1-\lambda)A_2,\lambda B_1+(1-\lambda)B_2)
			\]
			is strictly larger than
			\[
			\lambda\Phi_{p,q}^{(\theta)}(A_1,B_1)
			+
			(1-\lambda)\Phi_{p,q}^{(\theta)}(A_2,B_2).
			\]
			Let \(\mathcal F=\mathbb C^2\), with orthonormal basis
			\(\{|0\rangle,|1\rangle\}\). Define the block diagonal matrices
			\[
			\widehat A
			=
			\lambda |0\rangle\langle0|\otimes A_1
			+
			(1-\lambda)|1\rangle\langle1|\otimes A_2,
			\]
			and
			\[
			\widehat B
			=
			\lambda |0\rangle\langle0|\otimes B_1
			+
			(1-\lambda)|1\rangle\langle1|\otimes B_2.
			\]
			The space \(\mathcal F\) only records which branch of the convex
			combination is being used. Taking the partial trace over \(\mathcal F\)
			gives
			\[
			\Tr_{\mathcal F}\widehat A
			=
			\lambda A_1+(1-\lambda)A_2,
			\qquad
			\Tr_{\mathcal F}\widehat B
			=
			\lambda B_1+(1-\lambda)B_2.
			\]
			
			For every \(c>0\),
			\[
			M_r^{(\theta)}(cA,cB)=cM_r^{(\theta)}(A,B)
			\qquad (r\ge0),
			\]
			and hence
			\[
			\Phi_{p,q}^{(\theta)}(cA,cB)
			=
			c\Phi_{p,q}^{(\theta)}(A,B).
			\]
			Moreover, \(\Phi_{p,q}^{(\theta)}\) is additive over direct sums. Thus
			\[
			\Phi_{p,q}^{(\theta)}(\widehat A,\widehat B)
			=
			\lambda\Phi_{p,q}^{(\theta)}(A_1,B_1)
			+
			(1-\lambda)\Phi_{p,q}^{(\theta)}(A_2,B_2).
			\]
			Combining the preceding identities gives
			\[
			\Phi_{p,q}^{(\theta)}
			(\Tr_{\mathcal F}\widehat A,\Tr_{\mathcal F}\widehat B)
			>
			\Phi_{p,q}^{(\theta)}(\widehat A,\widehat B).
			\]
			Therefore the data processing inequality fails for the partial trace
			channel \(\Tr_{\mathcal F}\).
		\end{proof}
		
		\begin{corollary}\label{cor:sharpness-summary}
			For every \(0<\theta<1\), the data processing inequality fails in
			general in each of the following cases:
			\begin{enumerate}
				\item \(1<q<p\);
				\item \(p>2\) and \(q=1\);
				\item \(q=0\) and \(0<p<\max\{\theta,1-\theta\}\);
				\item more generally, any pair \(0<q<p<1\) with
				$
				\mathcal S_\theta(p,q)>0.
				$
			\end{enumerate}
		\end{corollary}
		
		\begin{proof}
			In case \((1)\), joint convexity fails by
			Theorem~\ref{thm:failure-q-greater-one}. In case \((2)\), it fails by
			Theorem~\ref{thm:p-greater-two-failure}. In case \((3)\), it fails by
			Theorem~\ref{thm:q0-scalar-classification}. In case \((4)\), the
			condition \(\mathcal S_\theta(p,q)>0\) gives a strict failure of scalar
			convexity by Theorem~\ref{thm:scalar-criterion-0qp1}. Thus in each case
			\(\Phi_{p,q}^{(\theta)}\) fails joint convexity. Applying
			Lemma~\ref{lem:flag-failure} gives failure of the data processing
			inequality in general.
		\end{proof}
		
		\section{Concluding remarks}\label{sec:concluding}
		
		We have introduced a weighted two-parameter family of trace gaps generated
		by matrix power means,
		\[
		\Phi_{p,q}^{(\theta)}(A,B)
		=
		\Tr\big[M_p^{(\theta)}(A,B)-M_q^{(\theta)}(A,B)\big],
		\qquad
		0<\theta<1,\quad p>q\ge0.
		\]
		We proved that every member of this family is a quantum divergence in the
		sense of Bhatia--Gaubert--Jain, including the logarithmic endpoint
		\(q=0\).
		
		In the Carlen--Lieb range
		\[
		0\le q\le1\le p\le2,
		\]
		we established joint convexity and the data processing inequality under
		completely positive trace-preserving maps. We also derived monotonicity
		with respect to the two parameters and computed the local quadratic form
		induced by the diagonal Hessian.
		
		Several familiar quantities occur as special cases or parameter limits.
		The subfamily \(p=1\) recovers arithmetic-versus-power-mean divergences of
		Lam--Le type; the point \((p,q)=(1,1/2)\) gives a constant multiple of the
		square-root Hellinger trace gap; and the point \((p,q)=(1,0)\) gives the
		log-Euclidean Hellinger-type trace gap. The normalized diagonal limit at
		\((p,q)=(1,1)\) recovers the weighted quantum Jensen--Shannon divergence,
		while the logarithmic endpoint limit \(p\downarrow0\), \(q=0\), yields an
		explicit logarithmic Hellinger-type expression.
		
		Finally, we showed that the convexity and data-processing range cannot be
		extended to all \(p>q\ge0\). Joint convexity fails for every \(1<q<p\),
		and for \(p>2\) failure already occurs at \(q=1\) through the Carlen--Lieb
		obstruction. In the remaining scalar region \(0<q<p<1\), we obtained an
		exact one-dimensional criterion in terms of \(\mathcal S_\theta(p,q)\).
		At the scalar logarithmic endpoint \(q=0\), the restriction is nonconvex
		exactly when
		\[
		0<p<\max\{\theta,1-\theta\}.
		\]
		A full matrix-level classification in the region \(0<q<p<1\) remains an
		interesting open problem.
		
		\section*{Declaration of Generative AI and AI-Assisted Technologies in the Writing Process}
		
		During the preparation of this work, the authors used ChatGPT (OpenAI) for language editing, including grammar checking, spelling correction, and stylistic improvements. The authors reviewed and edited the output as needed and take full responsibility for the content of the published article.

	\end{document}